\documentclass[a4paper,12pt]{article}
    \usepackage[top=2.5cm,bottom=2.5cm,left=2.5cm,right=2.5cm]{geometry}
    \usepackage{cite, amsmath, amssymb}
    \usepackage[margin=1cm,%
                font=small,%
                format=hang,%
                labelsep=period,%
                labelfont=bf]{caption}
\usepackage[latin1]{inputenc}
\usepackage{amsmath}
\usepackage{amsfonts}
\usepackage{amssymb}
\usepackage{cite}
\usepackage{amsthm}
\usepackage{makeidx}
\usepackage{graphicx}
\usepackage{mathrsfs,xcolor}
\usepackage{enumerate}
\usepackage{blkarray}
\usepackage{bm}
\usepackage{setspace}
\usepackage{float}
\usepackage{authblk}
\usepackage{multirow}

\numberwithin{table}{section}
\numberwithin{equation}{section}

\def \ds {\displaystyle}
\theoremstyle{plain}
\newtheorem{theorem}{Theorem}[section]
\newtheorem{proposition}[theorem]{Proposition}
\newtheorem{definition}[theorem]{Definition}
\newtheorem{lemma}[theorem]{Lemma}
\newtheorem{example}[theorem]{Example}

\newtheorem{remark}[theorem]{Remark}
\newtheorem{illustration}{Illustration}

\usepackage{authblk}
\author[1,2]{ \textbf{Daryl M. Magpantay}}
\author[3]{ \textbf{Bryan S. Hernandez}}
\author[3]{\textbf{Aurelio A. de los Reyes V}} 
\author[1,4,5,6]{\textbf{Eduardo R. Mendoza}}
\author[1,*]{ \textbf{Ederlina G. Nocon}}

\affil[1]{\small \textit{Mathematics and Statistics Department, De La Salle University, Manila  0922, Philippines}}
\affil[2]{\small \textit{College of Arts and Sciences, Batangas State University, Batangas City 4200, Philippines}}
\affil[3]{\small \textit{Institute of Mathematics, University of the Philippines Diliman, Quezon City 1101, Philippines}}
\affil[4]{\small \textit{Center for Natural Sciences and Environmental Research, De la Salle University, Manila 0922, Philippines}}
\affil[5]{\small \textit{Max Planck Institute of Biochemistry, Martinsried, Munich, Germany}}
\affil[6]{\small \textit{LMU Faculty of Physics, Geschwister -Scholl- Platz 1, 80539 Munich, Germany}}
\affil[*]{Corresponding author: \texttt{ederlina.nocon@dlsu.edu.ph}}

\title{\vspace{3.5cm}\textbf{A Computational Approach to Multistationarity in Poly-PL Kinetic Systems}}

\date{\normalsize (Received July 14, 2020)}

\begin{document}
\maketitle
\thispagestyle{empty}
\begin{abstract}
One important question that interests those who work in chemical reaction network theory (CRNT) is this:  \textit{Does the system obtained from a reaction network admit a positive equilibrium and if it does, can there be more than one within a stoichiometric class?}  The higher deficiency algorithm (HDA) of Ji and Feinberg provided a method of determining the multistationarity capacity of a CRN with mass action kinetics (MAK). An extension of this, called Multistationarity Algorithm (MSA), recently came into the scene tackling CRNs with power law kinetics (PLK), a kinetic system which is more general (having MAK systems as a special case). For this paper, we provide a computational approach to study the multistationarity feature of reaction networks endowed with kinetics which are non-negative linear combinations of power law functions called poly-PL kinetics (PYK).  The idea is to use MSA and combine it with a transformation called STAR-MSC 
(i.e., $S$-invariant Termwise Addition of Reactions via Maximal Stoichiometric Coefficients)
producing PLKs that are dynamically equivalent to PYKs. This leads us to being able to determinine the multistationarity capacity of a much larger class of kinetic systems.  We show that if the transformed dynamically equivalent PLK system is multistationary for a stoichiometric class for a set of particular rate constants, then so is its original corresponding PYK system.  Moreover, the monostationarity property of the transformed PLK system also implies the monostationarity property of the original PYK system.
\end{abstract}
\baselineskip=0.30in

\section{Introduction}
\label{intro}
The study of generalizations of mass action kinetics in chemical reaction network theory (CRNT) was initiated in 1972 by F. Horn and R. Jackson in their foundational paper entitled ``General Mass Action Kinetics'' \cite{HornJackson}. Their introduction of real coefficients in chemical reaction networks can be viewed as a geometric formalism for power law rate functions. This approach was further developed by S. M\"uller and G. Regensburger with their concept of ``generalized mass action system" in two contributions in 2012 and 2014 \cite{muller,MURE2014}. In parallel, C. Wiuf and E. Feliu conducted an extensive analysis of injectivity properties of power law and Hill-type kinetic systems in 2013 \cite{WIFE2013}. Furthermore, in 2015, Gabor et al. studied linear conjugacy of kinetic systems, which they called ``Bio-CRNs", with rational rate functions \cite{GABO2015}.

In power law kinetic systems, interactions between species are specified in a kinetic order matrix, which assigns a real number per species to each reaction. (In mass action systems, the kinetic orders are the stoichiometric coefficients of the reactants). This formalism allows additional aspects such as regulatory relationships or inhomogeneity of the reaction environment, which are particularly important in biochemical networks, to be considered. The power law approach to approximate biochemical system dynamics, pioneered by M. Savageau in 1969 \cite{SAVA1969a,SAVA1969b} has evolved into Biochemical Systems Theory (BST) with a multitude of applications to complex biochemical systems (cf. E. Voit's review  \cite{VOIT2013}). A chemical reaction network approach specifically for BST models was developed by Arceo et al. in 2015 and 2017 \cite{AJMM2015,AJKM2017}. Voit et al. \cite{VOMO2015} provides a compact overview of the evolution of the field from the work of C. Guldberg and P. Waage  in 1865 through Savageau's 1969 power law approach to the connection with CRNT by M\"uller and Regensburger in 2012.

Poly-PL kinetic systems (denoted as PYK systems) are chemical reaction networks (CRNs) endowed with non-negative linear combinations of power law functions.
These systems were introduced by Talabis et al. \cite{tmmrn2020} and were shown to have complex balanced equilibria for weakly reversible such systems with zero kinetic reactant deficiency, which are called PY-TIK systems. The PYK subset of polynomial kinetic systems (POK systems) occur in realizations of evolutionary games as chemical kinetic systems proposed by Veloz et al. \cite{veloz2014}, particularly for multi-player games with replicator dynamics.

This paper addresses the problem of multistationarity in a poly-PL system: are there rate constants such that two distinct positive equilibria in a stoichiometric class exist?

In 2011, the Higher Deficiency Algorithm (HDA) for mass action kinetic (MAK) systems, which examines the systems' capacity for multistationarity (i.e., multiple equilibria), was introduced in the PhD thesis of H. Ji under the mentorship of M. Feinberg \cite{ji}. The algorithm was based on the concepts of fundamental classes of underlying reactions together with ideas such as the upper-middle-lower partitions familiar from the Deficiency One Algorithm. It was extended to the set of power law kinetic systems with reactant-determined kinetics (PL-RDK) by Hernandez et al. \cite{hmr2019}, and shown to be applicable to PL-NDK systems (i.e., there are nodes with branching reactions with different interactions for the kinetic order rows), if the fundamental decomposition generated by the fundamental classes is independent (the network's stoichiometric subspace is the direct sum of the subnetworks' stoichiometric subspaces) \cite{hmr22019}. Otherwise, a method called the CF-RM transformation introduced in \cite{neml2019} must first be applied to get a dynamically equivalent system (also called ``realization'' of the original system) with reactant-determined kinetics (PL-RDK), and then applies the extended HDA to determine its capacity for multistationarity.
This procedure was called the Multistationarity Algorithm (MSA) for power law kinetic systems (MSA-PLK).

STAR methods construct realizations of the original system with the same stoichiometric subspace (ST = S-invariant Transformation) by adding reactions (=AR). The STAR-MSC variant is based on the idea of using the maximal stoichiometric coefficient (MSC) among the complexes in CRNs to construct reactions whose reactant complexes and product complexes are different from existing ones. This is done by uniform translation of the reactants and products to create a ``replica'' of the CRN. The method creates $h - 1$  replicas of the original network $\mathscr{N}$. Hence, the transformed network $\mathscr{N}^*$ of $\mathscr{N}$ becomes the union, in the sense of \cite{ghms2019}, of the replicas and the original CRN.

This paper highlights the combination of the ``STAR-MSC transformation'' of a poly-PL kinetic system to a dynamically equivalent power law kinetic system introduced in \cite{dmag2019} and the MSA-PLK of Hernandez et al. \cite{hmr2019} to determine the capacity of the poly-PL kinetic system for multistationarity, i.e., the existence of different positive equilibria in a stoichiometric class for certain set of rate constants. 
We illustrate the steps of the method with a running example, which has the capacity for multistationarity within a stoichiometric class.

Fortun et al. \cite{FTJM2020} recently used the canonical PL-representation of a PYK system from \cite{tmmrn2020} together with STAR-MSC to extend results on generalized mass action kinetic (GMAK) systems, including those of M\"uller and Regensburger \cite{muller,MURE2014} on complex equilibria multiplicity and of Boros et al. \cite{BOMR2019} on linear stability to subsets of PY-RDK systems. They also extend results of Fortun and Mendoza \cite{FOME2020} on concentration robustness to these subsets. Furthermore, Hernandez and Mendoza \cite{HEME2020} studied Hill-type systems, which are widely used in enzyme kinetics, by associating a unique PYK system to any Hill-type system and obtained results on multiplicity and concentration robustness for such systems. The results of this paper enable computational approaches to aspects of these applications of PYK systems.

The paper is organized as follows: Section \ref{prelim} introduces the fundamental concepts and results on chemical reaction networks and kinetic systems needed for later sections. Section \ref{PYK:STAR:MSC} reviews the basic properties of poly-PL systems and illustrates them with a running example. It then introduces the STAR-MSC transformation in detail, supporting the discussion with computations for the running example. In Section \ref{application:MSA:STAR}, after a review of MSA-PLK, the simple criterion for determining the subset of equilibria of the poly-PL system is presented. The computations for the running example are presented to complete the illustration of ``MSA-PYK.'' Section \ref{sec:conclusion} summarizes the results of the paper and provides an outlook for further work.

\section{Fundamentals of Chemical Reaction Networks and Kinetic Systems}
\label{prelim}
In this section, we discuss fundamental concepts and results about chemical reaction networks (CRN) and chemical kinetic systems (CKS). Moreover, we also explore CRN as a digraph with vertex labeling. Also, we focus on CKS side of power-law kinetics (PLK) system.

\begin{definition}
A \textbf{chemical reaction network (CRN)} is a digraph $(\mathscr{C},\mathscr{R})$ where each vertex has positive degree and stoichiometry, i.e., there is a finite set $\mathscr{S}$ (whose elements are called \textbf{species}) such that $\mathscr{C}$ is a subset of $\mathbb{R}^{\mathscr{S}}_{\geq 0}$. Each vertex is called a \textbf{complex} and its coordinates in $\mathbb{R}^{\mathscr{S}}_{\geq 0}$ are called \textbf{stoichiometric coefficients}. The arcs are called \textbf{reactions}.
\end{definition}

We denote the number of species with $m$, the number of complexes with $n$, and the number of reactions with $r$. Also, we denote this nonempty finite collection of reactions as $\mathscr{R} \subset (\mathscr{C} \times \mathscr{C})$. We implicitly assume the sets are numbered and let
\begin{displaymath}
\mathscr{S}=\{X_1, X_2, \ldots, X_m\}, \hspace{0.5cm} \mathscr{C}=\{C_1, C_2, \ldots, C_n\} \;\;\; \text{and}\;\;\; \mathscr{R}=\{R_1, R_2, \ldots, R_r\},
\end{displaymath}
where $m,n$ and $r$ are their respective cardinalities. Thus, $\mathbb{R}^{\mathscr{S}}_{\geq 0} \cong \mathbb{R}^m_{\geq 0}$. Consider the reaction
\begin{displaymath}
\alpha X_1 + \beta X_2 \rightarrow \gamma X_3,
\end{displaymath}
$X_1,X_2$ and $X_3$  are the species. The complexes are $\alpha X_1+\beta X_2$  and $\gamma X_3$. In particular, $\alpha X_1+\beta X_2$ is called the \textbf{reactant} (or \textbf{source}) \textbf{complex} and $\gamma X_3$ the \textbf{product complex}. The number of reactant complexes is denoted by $n_r$.

The stoichiometric coefficients are the non-negative integer coefficients $\alpha, \beta$ and $\gamma$. Under mass action kinetics (MAK), the rate at which the reaction occurs is given by the monomial
\begin{displaymath}
K=kX^{\alpha}_1 X^{\beta}_2
\end{displaymath}
with rate constant $k>0$. We can generalize this by considering power-law kinetics. The reaction rate can be
\begin{displaymath}
K=kX^a_1 X^b_2
\end{displaymath}
where $a$ and $b$ can be any real number. We call $a$ and $b$ as \textbf{kinetic orders}. Within a network involving additional species and reactions, the above reaction contributes to the dynamics of the species concentration as
\begin{displaymath}
	\dot{X}=
	 \left[
	\begin{array}{c}
	\dot{X}_1\\
	\dot{X}_2\\
	\dot{X}_3\\
	\vdots
	\end{array}
	\right]= kX^a_1 X^b_2 \left(
	\begin{array}{c}
	-\alpha\\
	-\beta\\
	\gamma\\
	\vdots
	\end{array}
	\right)+ \cdots .
	\end{displaymath}

\begin{definition}
Let $\mathscr{N}=(\mathscr{S,C,R})$ be a CRN. The {\textbf{incidence map}} $I_a : \mathbb{R}^\mathscr{R} \rightarrow \mathbb{R}^\mathscr{C}$ is the linear map such that for each reaction $r:C_i \rightarrow C_j \in \mathscr{R}$, the basis vector $\omega_r$ to the vector $\omega_{C_j}-\omega_{C_i} \in \mathscr{C}$.
\end{definition} 

\begin{definition}
The {\textbf{stoichiometric subspace}} of a reaction network $\left(\mathscr{S},\mathscr{C},\mathscr{R}\right)$, denoted by $S$, is the linear subspace of $\mathbb{R}^\mathscr{S}$ given by $S = span\left\{ {{C_j} - {C_i} \in \mathbb{R}^\mathscr{S}|\left( {{C_i},{C_j}} \right) \in \mathscr{R}} \right\}.$ The {\textbf{rank}} of the network, denoted by $s$, is given by $s=\dim S$. The set $\left( {x + S} \right) \cap \mathbb{R}_{ \ge 0}^\mathscr{S}$ is said to be a {\textbf{stoichiometric compatibility class}} of $x \in \mathbb{R}_{ \ge 0}^\mathscr{S}$.
\end{definition}

\begin{definition}
Two vectors $x, x^{*} \in {\mathbb{R}^\mathscr{S}}$ are {\textbf{stoichiometrically compatible}} if $x-x^{*}$ is an element of the stoichiometric subspace $S$.
\end{definition}

\begin{definition}
The \textbf{linkage classes} of a CRN are the subnetworks of a reaction graph where for any complexes $C_i,C_j$  of the subnetwork, there is path between them. The number of linkage classes is denoted by $\ell$.
\end{definition}
The linkage class is said to be a \textbf{strong linkage class} if there is a directed path from $C_i$ to $C_j$ and vice versa for any complexes $C_i, C_j$ of the subnetwork. The number of strong linkage classes is denoted by $s\ell$. Moreover, \textbf{terminal strong linkage classes}, the number of which is denoted as $t$, is the maximal strongly connected subnetworks where there are no edges (reactions) from a complex in the subgraph to a complex outside the subnetwork. The terminal strong linkage classes can be of two kinds: cycles (not necessarily simple) and singletons (which we call ``terminal points'').

\begin{example}
Given a chemical reaction network (CRN):

\hspace{1in} $R_1: 2A_1 \rightarrow A_3$

\hspace{1in} $R_2: A_2+A_3 \rightarrow A_3$

\hspace{1in} $R_3: A_3 \rightarrow A_2+A_3$

\hspace{1in} $R_4: 3A_4 \rightarrow A_2+A_3$

\hspace{1in} $R_5: 2A_1 \rightarrow 3A_4$

As observed, $m=4$ (species), $n=4$ (complexes), $n_r=4$ (reactant complexes) and $r=5$ (reactions). Also, we have

\hspace{1in} $\mathscr{S}=\{A_1,A_2,A_3,A_4\}$

\hspace{1in} $\mathscr{C}=\{C_1=2A_1, C_2=A_2+A_3, C_3=A_3, C_4=3A_4\}$

The number of linkage classes is one ($\ell=1$): $\{2A_1, A_3, A_2+A_3, 3A_4\}$, 
strong linkage classes is three ($s\ell=3$): $\{A_3, A_2+A_3\}$, $\{2A_1\}$ , $\{3A_4\}$ and
terminal strong linkage class is one ($t\ell=1$) : $\{A_3, A_2+A_3\}$.
\end{example}

\begin{definition}
A CRN is called
\begin{enumerate}[1.]
\item \textbf{weakly reversible} if $s\ell = \ell$;
\item \textbf{t-minimal} if $t=\ell$;
\item \textbf{point-terminal} if $t=n-n_r$; and
\item \textbf{cycle terminal} if $n-n_r=0$.
\end{enumerate}
\end{definition}

As observed in our running example, $t=1$ and $n-n_r=4-4=0$. This implies that the network is not point terminal. Also, $s\ell=3 \neq 1=\ell$. Hence, the network is not weakly reversible. But, $t=1=\ell$. Thus, the network is $t$-minimal.

\begin{definition}
The \textbf{deficiency of a CRN} is the integer $\delta = n-\ell-s$.
\end{definition}

In the running example, observe that $n=4$, $\ell=1$ and $s=3$. With this, the deficiency of the network $\delta=n-\ell-s=4-1-3=0$.

We now introduce useful partition of networks with respect to terminality:

\begin{definition}
A CRN is of type \textbf{terminality bounded by deficiency
(TBD)} if $t-\ell \leq \delta$, otherwise, it is of type \textbf{terminality not deficiency-bounded
(TND)}, i.e. $t-\ell > \delta$. 
\end{definition}

\begin{definition}
A CRN has \textbf{low reactant diversity (LRD)} if $n_r < s$, otherwise it has \textbf{sufficient reactant diversity (SRD)}. An SRD network has \textbf{high reactant diversity (HRD)} or \textbf{medium reactant diversity (MRD)} if $n_r > s$ or $n_r= s$, respectively.
\end{definition}

Going back to the running example, we have $t-\ell=1-1=0=\delta$ and $n_r=4 > 3=s$. Hence, the CRN is of type terminality bounded by deficiency (TBD) and has sufficient reactant diversity (SRD).

\begin{definition}
The \textbf{reactant subspace} $R$ is the linear space in $\mathbb{R}^{\mathscr{S}}$ generated by the reactant complexes, i.e., $\langle \rho(\mathscr{R})\rangle$.  The value $q:= dim\;R$ is called the \textbf{reactant
rank} of the network.
\end{definition}

In \cite{REACT2018}, the CRNs were classified based on the subspace $R\cap S$. The following are two interesting subsets of nontrivial intersection:

\begin{definition}
A CRN has a \textbf{stoichiometry-determined reactant subspace} (of type \textbf{SRS}) if its nonzero reactant subspace $R$ is contained in $S$, i.e. $0\neq R=R\cap S$. It has a \textbf{reactant-determined stoichiometric subspace} (of type \textbf{RSS}) if $S$ is contained in $R$, i.e., $R \cap S = S$. A CRN in SRS $\cap$ RSS has type \textbf{RES  (Reactant Subspace Equal to Stoichiometric Subspace)}, i.e., $R = S$.
\end{definition}

The dynamical system of the CRN of our running example can be written as

\begin{center}
	$\begin{array}{ccccc}
	R_1 & R_2 & R_3 & R_4&R_5\;\;\;\;\;\;\;\;\;\;\;\;\;\;\;\;\;
	\end{array}$\\
	$\dot{X}=
	 \left[
	\begin{array}{c}
	\dot{A_1}\\
	\dot{A_2}\\
	\dot{A_3}\\
	\dot{A_4}
	\end{array}
	\right]=  \left[
	\begin{array}{ccccc}
	-2& 0 &0&0 & -2\\
	0 & -1 & 1 & 1 & 0\\
	1 & 0 & 0 & 1 & 0\\
	0 & 0 & 0 &-3 & 3
	\end{array}	\right]
	 \left[
	\begin{array}{c}
	k_{13}A_1^{f_{11}}\\
	k_{23}A_2^{f_{21}}A_3^{f_{22}}\\
	k_{32}A_3^{f_{31}}\\
	k_{42}A_4^{f_{41}}\\
	k_{14}A_1^{f_{51}}
	\end{array}	\right]=NK(x).$
	\end{center}
$N$ is called the stoichiometric matrix and $K(x)$ is called the \textbf{kinetic vector} (or \textbf{kinetics}). The \textbf{chemical kinetic system (CKS)} is defined as follows.

\begin{definition}
A \textbf{kinetics} of a CRN $\mathscr{N}=(\mathscr{S, C, R})$ is an assignment of a rate function $K_{ij}: \Omega_K \rightarrow \mathbb{R}_{\geq 0}$  to each reaction $(i, j) \in \mathscr{R}$, where $\Omega_K$ is a set such that $\mathbb{R}^{\mathscr{S}}_{>0} \subseteq \Omega_K \subseteq \mathbb{R}^{\mathscr{S}}_{\geq 0}$, and
\begin{displaymath}
K_{ij}(c)\geq 0, \;\;\; for\;all\;c \in \Omega_K.
\end{displaymath}
A kinetics for a network $\mathscr{N}$ is denoted by $K=[K_1,K_2,\ldots,K_r]^T: \Omega_K \rightarrow \mathbb{R}^{\mathscr{S}}_{\geq 0}$.  The pair $(\mathscr{N} ,K)$ is called the \textbf{chemical kinetic system} (CKS).
\end{definition}

A chemical kinetics gives rise to two closely related objects: the species formation rate function (SFRF) and the associated ODE systems:

\begin{definition}
The \textbf{species formation rate function (SFRF)} of CKS is the vector field
\begin{displaymath}
f(x)=NK(x)= \sum_{y\rightarrow y'} K_{y \rightarrow y'} (x) (y'-y).
\end{displaymath}
where $N$ is the stoichiometric matrix. The equation $\dot{x}=f(x)$ is the \textbf{ODE} or \textbf{dynamical system} of the CKS.
\end{definition}

\begin{definition}
The {\textbf{set of positive equilibria}} of a chemical kinetic system $\left(\mathscr{N},K\right)$ is given by ${E_ + }\left(\mathscr{N},K\right)= \left\{ {x \in \mathbb{R}^\mathscr{S}_{>0}|f\left( x \right) = 0} \right\}.$
\end{definition}

Analogously, the {\bf set of complex balanced equilibria} \cite{HornJackson} is given by 
\[{Z_ + }\left(\mathscr{N},K\right) = \left\{ {x \in \mathbb{R}_{ > 0}^\mathscr{S}|{I_a} \cdot K\left( x \right) = 0} \right\} \subseteq {E_ + }\left(\mathscr{N},K\right).\]
A positive vector $c \in \mathbb{R}^\mathscr{S}$ is complex balanced if $K\left( c \right)$ is contained in ${\text{Ker }}{I_a}$, and a chemical kinetic system is complex balanced if it has a complex balanced equilibrium.

The ODE system is under \textbf{power law kinetics (PLK)} which has the form
\begin{displaymath}
K_i(x)=k_i \prod^m_{j=1} x^{F_{ij}} \;\;\; {\text{where}} \; 1 \leq i \leq r
\end{displaymath}
with $k_i \in \mathbb{R_+}$ and $F_{ij} \in \mathbb{R}$. Power law kinetics is defined by an $r \times m$ matrix $F = [F_{ij}]$, called the \textbf{kinetic order matrix}, and vector $k \in \mathbb{R}^r$, called the \textbf{rate vector}. A particular example of power law kinetics is the well-known mass action kinetics where the kinetic order matrix consists of stoichiometric coefficients of the reactants. In the running example, we assume power law kinetics so that the kinetic order matrix is


\begin{center}
$	\begin{array}{cccc}
	A_1 & A_2 & A_3 & A_4   
	\end{array}$\\
$F=	 \left[\begin{array}{cccc}
	F_{11} & 0 & 0& 0 \\
	0 & F_{22} & F_{32} & 0\\
	0 & 0 & F_{33} & 0  \\
	0 & 0 & 0 & F_{44} \\
	F_{15} & 0 & 0 & 0
	\end{array}
	\right]$
$\begin{array}{c}
	 R_1\\
	R_2\\
	R_3  \\
	R_4\\
	R_5
	\end{array}
$
\end{center}

\noindent where $F_{ij} \in \mathbb{R}$. 

\begin{definition}
A PLK system has {\textbf{reactant-determined kinetics}} (of type PL-RDK) if for any two reactions $i, j$ with identical reactant complexes, the corresponding rows of kinetic orders in $F$ are identical, i.e., ${F_{ik}} = {F_{jk}}$ for $k = 1,2,...,m$. A PLK system has {\textbf{non-reactant-determined kinetics}} (of type PL-NDK) if there exist two reactions with the same reactant complexes whose corresponding rows in $F$ are not identical.
\end{definition}

\subsection{Fundamentals of Decomposition Theory}
\label{sect:decomposition}

This subsection recalls some definitions and earlier results from the decomposition theory of chemical reaction networks. A more detailed discussion can be found in \cite{FAML2020}.

\begin{definition}
A {\textbf{decomposition}} of $\mathscr{N}$ is a set of subnetworks $\{\mathscr{N}_1, \mathscr{N}_2,...,\mathscr{N}_k\}$ of $\mathscr{N}$ induced by a partition $\{\mathscr{R}_1, \mathscr{R}_2,...,\mathscr{R}_k\}$ of its reaction set $\mathscr{R}$. 
\end{definition}

We denote a decomposition by 
$\mathscr{N} = \mathscr{N}_1 \cup \mathscr{N}_2 \cup ... \cup \mathscr{N}_k$
since $\mathscr{N}$ is a union of the subnetworks in the sense of \cite{ghms2019}. It also follows immediately that, for the corresponding stoichiometric subspaces, 
${S} = {S}_1 + {S}_2 + ... + {S}_k$.

The following important concept of independent decomposition was introduced by Feinberg in \cite{feinberg12}.

\begin{definition}
A network decomposition $\mathscr{N} = \mathscr{N}_1 \cup \mathscr{N}_2 \cup ... \cup \mathscr{N}_k$  is {\textbf{independent}} if its stoichiometric subspace is a direct sum of the subnetwork stoichiometric subspaces.
\end{definition}

It was shown that for an independent decomposition, $\delta \le \delta_1 +\delta_2 ... +\delta_k$ \cite{fortun2}.

\begin{definition} \cite{FAML2020}
A decomposition $\mathscr{N} = \mathscr{N}_1 \cup \mathscr{N}_2 \cup ... \cup \mathscr{N}_k$ with $\mathscr{N}_i = (\mathscr{S}_i,\mathscr{C}_i,\mathscr{R}_i)$ is a {\textbf{$\mathscr{C}$-decomposition}} if for each pair of distinct $i$ and $j$, $\mathscr{C}_i$ and $\mathscr{C}_j$ are disjoint.
\end{definition}

\begin{definition}
A decomposition of CRN $\mathscr{N}$ is {\textbf{incidence-independent}} if the incidence map $I_a$ of $\mathscr{N}$ is the direct sum of the incidence maps of the subnetworks. It is {\textbf{bi-independent}} if it is both independent and incidence-independent.
\end{definition}

We can also show incidence-independence by satisfying the equation $$n - l = \sum {\left( {{n_i} - {l_i}} \right)},$$ where $n_i$ is the number of complexes and $l_i$ is the number of linkage classes, in each subnetwork $i$.

In \cite{FAML2020}, it was shown that for any incidence-independent decomposition, $\delta \ge \delta_1 +\delta_2 ... +\delta_k$ and that $\mathscr{C}$-decompositions form a subset of incidence-independent decompositions.

Feinberg established the following basic relation between an independent decomposition and the set of positive equilibria of a kinetics on the network:

\begin{theorem} (Feinberg Decomposition Theorem \cite{feinberg12})
\label{feinberg:decom:thm}
Let $P(\mathscr{R})=\{\mathscr{R}_1, \mathscr{R}_2,...,\mathscr{R}_k\}$ be a partition of a CRN $\mathscr{N}$ and let $K$ be a kinetics on $\mathscr{N}$. If $\mathscr{N} = \mathscr{N}_1 \cup \mathscr{N}_2 \cup ... \cup \mathscr{N}_k$ is the network decomposition of $P(\mathscr{R})$ and ${E_ + }\left(\mathscr{N}_i,{K}_i\right)= \left\{ {x \in \mathbb{R}^\mathscr{S}_{>0}|N_iK_i(x) = 0} \right\}$ then
\[{E_ + }\left(\mathscr{N}_1,K_1\right) \cap {E_ + }\left(\mathscr{N}_2,K_2\right) \cap ... \cap {E_ + }\left(\mathscr{N}_k,K_k\right) \subseteq  {E_ + }\left(\mathscr{N},K\right).\]
If the network decomposition is independent, then equality holds.
\end{theorem}

The analogue of Feinberg's 1987 result for incidence-independent decompositions and complex balanced equilibria is shown in \cite{FAML2020}:

\begin{theorem} (Theorem 4 \cite{FAML2020})
\label{decomposition:thm:2}
Let $\mathscr{N}=(\mathscr{S},\mathscr{C},\mathscr{R})$ be a a CRN and $\mathscr{N}_i=(\mathscr{S}_i,\mathscr{C}_i,\mathscr{R}_i)$ for $i = 1,2,...,k$ be the subnetworks of a decomposition.
Let $K$ be any kinetics, and $Z_+(\mathscr{N},K)$ and $Z_+(\mathscr{N}_i,K_i)$ be the set of complex balanced equilibria of $\mathscr{N}$ and $\mathscr{N}_i$, respectively. Then
\begin{itemize}
\item[i.] ${Z_ + }\left(\mathscr{N}_1,K_1\right) \cap {Z_ + }\left(\mathscr{N}_2,K_2\right) \cap ... \cap {Z_ + }\left(\mathscr{N}_k,K_k\right) \subseteq  {Z_ + }\left(\mathscr{N},K\right)$.\\
If the decomposition is incidence independent, then
\item[ii.] ${Z_ + }\left( {\mathscr{N},K} \right) = {Z_ + }\left(\mathscr{N}_1,K_1\right) \cap {Z_ + }\left(\mathscr{N}_2,K_2\right) \cap ... \cap {Z_ + }\left(\mathscr{N}_k,K_k\right)$, and
\item[iii.] ${Z_ + }\left( {\mathscr{N},K} \right) \ne \varnothing$ implies ${Z_ + }\left( {\mathscr{N}_i,K_i} \right) \ne \varnothing$ for each $i=1,...,k$.
\end{itemize}
\end{theorem}

The converse of Theorem \ref{decomposition:thm:2} iii holds for a subset of incidence-independent decompositions with any given kinetics.

\begin{theorem} (Theorem 5 \cite{FAML2020})
Let $\mathscr{N} = \mathscr{N}_1 \cup \mathscr{N}_2 \cup ... \cup \mathscr{N}_k$ be a weakly reversible $\mathscr{C}$-decomposition of a chemical kinetic system $(\mathscr{N},K)$. If ${Z_ + }\left( {\mathscr{N}_i,K_i} \right) \ne \varnothing$ for each $i=1,...,k$, then ${Z_ + }\left( {\mathscr{N},K} \right) \ne \varnothing$.
\end{theorem}

\bigskip

\section{Poly-PL Kinetic Systems and their Transformations to PL Kinetic Systems }
\label{PYK:STAR:MSC}

In this paper, we are more interested in a kinetic system composed of
non-negative linear combinations of power law functions.

\textbf{Poly-PL kinetics (PYK) }are kinetic systems consisting of non-negative linear combinations of
power law functions. This set contains the set PLK of power law kinetics as ``mono-PL kinetics with
coefficient 1''. Like PLK, the domain of PYK is the positive orthant $\mathbb{R}^m_{> 0}$. However, for subsets, this may be extended to the whole non-negative orthant $\mathbb{R}^m_{\geq 0}$. Clearly, PYK and PLK generate the same sets of SFRFs, the power law dynamical systems (or GMA systems in BST terminology).

After setting the standard ordering of species $X_1, \ldots, X_m$, we have the following definition:

\begin{definition}
A kinetics $K : \mathbb{R}^m_{>0} \rightarrow  \mathbb{R}^r$ is a \textbf{poly-PL kinetics} if
\begin{equation}\label{ppkeq}
K_i(x)= k_i(a_{i,1}x^{F_{i,1}}+ \ldots + a_{i,j}x^{F_{i,j}}) \;\;{\text{where}}\; 1 \leq i \leq r
\end{equation}
written in lexicographic order with $k_i >0$, $F_{i,j}, a_{i,j} \in \mathbb{R}^m$ and $1 \leq j \leq h_i$ (where $h_i$ is the number of terms in reaction $i$). Power-law kinetics is defined by $r \times m$ matrices $F_{i,k} = [F_{ij}]$, called the \textbf{kinetic order} matrices, vectors $k = [k_i]$ called the \textbf{rate vector} and $a_{i,j}  \in \mathbb{R}^r_{> 0}$ called the \textbf{poly-rate vectors}.
\end{definition}

\begin{example}
For $(\mathscr{N},K)$ with $\mathscr{S} =\{X,Y\}$ and $\mathscr{R}=\{r:X\rightarrow2X,r':2X\rightarrow 5X+Y\}$, let the poly-PL kinetics be given by :
\begin{displaymath}
K_1 (X,Y)=k_1 (2XY+0.5Y^2 )\; \hspace{1cm}
K_2 (X,Y)=k_2 (0.75X^2 Y+X^3) 
\end{displaymath}
where $k_1$ and $k_2$ are rate constants. The kinetic order matrices are

\begin{center}
$F_{1,k}=	 \left[\begin{array}{cc}
	1 & 1 \\
	2 & 1
	\end{array}
	\right]$; \hspace{0.5cm} $F_{2,k}=	 \left[\begin{array}{cc}
	0 & 2 \\
	3 & 0
	\end{array}
	\right]$ 
\end{center}

with the kinetic order vector

\begin{center}
$k=	 \left[\begin{array}{c}
	k_1 \\
	k_2
	\end{array}
	\right]$.
\end{center}
\end{example}

STAR ($S$-invariant Termwise Addition of Reactions) is a network structure-oriented approach to poly-PL kinetics based on the following basic observation: for the rate function $K_i(x)$ and for a reaction $r_i: y_i \rightarrow y_i'$ in a PYK system $(\mathscr{N},K)$ with $\mathscr{N}=(\mathscr{S}, \mathscr{C}, \mathscr{R})$ we have
\begin{equation}\label{keq}
K_i(x)=k_i(a_{i1}M_{i1}+ \ldots + a_{ih}M_{ih})(y_i'-y_i)= k_i a_{i1}M_{i1}(y_i'-y_i)+ \ldots + k_i a_{ih}M_{ih} (y_i'-y_i)
\end{equation}
where $M_{ij}$ are the $h$ power law functions for the $i$-th reaction. 

A STAR method introduces additional different reaction(s) for each of the $h$ identical reaction vectors $y_i'-y_i$ in the sum. This enlarges the sets of reactions and complexes, so the new CRN $\mathscr{N}^*=(\mathscr{S}, \mathscr{C}^*, \mathscr{R}^*)$ and new kinetics $K^*: \mathbb{R}_{> 0}^{\mathscr{S}} \rightarrow \mathbb{R}^{\mathscr{R}^*}$ are constructed.

\begin{remark}
When the poly-PL kinetics do not have the same number of terms, one simply uses the same ``trick'' of replacing the last term of the shorter function with $(h -h'+ 1)$ copies of $\dfrac{1}{h - h' + 1}$ of that term. This representation of the poly-PL system is called the canonical PL-representation of the system.
\end{remark}

\begin{illustration}
Suppose the following are the poly-PL kinetics of a network:
\begin{displaymath}
K_1 (X,Y)=k_1(2X^3Y+X^2Y^2+ 4XY^2+0.5Y^3)
\end{displaymath}
\begin{displaymath}
K_2 (X,Y)=k_2(5X^4 Y^2+6XY^3).
\end{displaymath}
The above kinetics is equivalent to: 
\begin{displaymath}
K_1 (X,Y)=k_1(2X^3Y+X^2Y^2+ 4XY^2+0.5Y^3 )
\end{displaymath}
\begin{displaymath}
K_2 (X,Y)=k_2(5X^4 Y^2+2XY^3+2XY^3+2XY^3).
\end{displaymath}
\end{illustration}
Notice that $\mathscr{N}$ and $\mathscr{N}^*$ have the same set of species as one of requirements to be dynamically equivalent. Since it is $S$-invariant and it is a transformation, this assures the equality of the stoichiometric subspace  $S$ of the original network to the stoichiometric subspace $S^*$ of the network produced by adding complexes and reactions. This implies, $s= \dim S = \dim S^*=s^*$.

Aside from these observations, we have the following properties for any STAR method:

\begin{proposition}
Let $\mathscr{N}^*=(\mathscr{S}, \mathscr{C}^*, \mathscr{R}^*)$ be a STAR transform of $\mathscr{N}=(\mathscr{S}, \mathscr{C}, \mathscr{R})$. Then  $|\mathscr{R}^*|=hr$ and $|\mathscr{C}^*| \leq hn$.
\end{proposition}
\begin{proof}
For the first part, since STAR method introduces an additional different reaction for each of the $h$ identical reaction vectors $y_i'- y_i$ in the sum then claim follows immediately from (\ref{keq}) and the fact that $hn$ is the maximum number of complexes (i.e., when these are all different) for $hr$ reactions.
\end{proof}

\begin{remark} 
Observe that $K^*_{ij}=k^*_{ij}M_{ij}$, with $k^*_{ij}=k_ia_{ij}$, is the kinetic function of the reaction $R_{ij}$ corresponding to the reaction vector $(y_i' -y_i)$. The stoichiometric matrix $N^*$ becomes an $m \times hr$ matrix, and usually the reaction vector of the reaction $r_i$ are just replicated $h$ times as assumed in (ii).
\end{remark}

Let the reaction vector $y_i'- y_i=(c_{i1}, c_{i2}, \ldots, c_{im})$. By the remark above, we have

\begin{center}
$	\begin{array}{ccccccccccccc}
	R_{11}& R_{12} & \ldots & R_{1h}&  R_{21}& R_{22} & \ldots & R_{2h}& \ldots&  R_{r1}& R_{r2} & \ldots & R_{rh}
	\end{array}$\\
$N^*=	 \left[\begin{array}{ccccccccccccc}
	c_{11} & c_{11}& \ldots & c_{11}& c_{21} & c_{21}& \ldots & c_{21} & \ldots & c_{r1} & c_{r1}& \ldots & c_{r1} \\
	c_{12} & c_{12}& \ldots & c_{12}& c_{22} & c_{22}& \ldots & c_{22} & \ldots & c_{r2} & c_{r2}& \ldots & c_{r2} \\
	\vdots &\vdots& \ddots & \vdots& \vdots & \vdots& \ddots & \vdots & \ddots & \vdots & \vdots& \ddots & \vdots \\
	c_{1m} & c_{1m}& \ldots & c_{1m}& c_{2m} & c_{2m}& \ldots & c_{2m} & \ldots & c_{rm} & c_{rm}& \ldots & c_{rm}
	\end{array}
	\right]$
$\begin{array}{c}
	 A_1\\
	A_2\\
	\vdots  \\
	A_m
	\end{array}
$ 
\end{center}
and

\begin{center}
$K^*=	 \left[\begin{array}{c}
	K^*_{11} \\
	K^*_{12} \\
	\vdots\\
	K^*_{1h} \\
	K^*_{21} \\
	K^*_{22} \\
	\vdots\\
	K^*_{2h} \\
	\vdots\\
	K^*_{r1} \\
	K^*_{r2} \\
	\vdots\\
	K^*_{rh} \\
	\end{array}
	\right]$.

\end{center}

Hence,

\begin{align*}
N^*K^*&= \left[\begin{array}{c}
	c_{11}K^*_{11}+ c_{11}K^*_{12}+\ldots+c_{11}K^*_{1h}+ 
\ldots + c_{r1}K^*_{r1}+ c_{r1}K^*_{r2}+\ldots+c_{r1}K^*_{rh}\\
	c_{12}K^*_{11}+ c_{12}K^*_{12}+\ldots+c_{12}K^*_{1h}+ \ldots + c_{r2}K^*_{r1}+ c_{r2}K^*_{r2}+\ldots+c_{r2}K^*_{rh}\\
	\vdots\\
	c_{1m}K^*_{11}+ c_{1m}K^*_{12}+\ldots+c_{1m}K^*_{1h}+ \ldots + c_{rm}K^*_{r1}+ c_{rm}K^*_{r2}+\ldots+c_{rm}K^*_{rh}
	\end{array}
	\right]\\
&=	 \left[\begin{array}{c}
	c_{11}(K^*_{11}+ K^*_{12}+\ldots+K^*_{1h})+ 
\ldots + c_{r1}(K^*_{r1}+ K^*_{r2}+\ldots+K^*_{rh})\\
	c_{12}(K^*_{11}+ K^*_{12}+\ldots+K^*_{1h})+ \ldots + c_{r2}(K^*_{r1}+ K^*_{r2}+\ldots+K^*_{rh})\\
	\vdots\\
	c_{1m}(K^*_{11}+ K^*_{12}+\ldots+K^*_{1h})+ \ldots + c_{rm}(K^*_{r1}+ K^*_{r2}+\ldots+K^*_{rh})
	\end{array}
	\right]\\
&=	 \left[\begin{array}{c}
	c_{11}K_1+ c_{21}K_2 + \ldots +c_{r1}K_r\\
	c_{12}K_1+ c_{22}K_2 + \ldots + c_{r2}K_r\\
	\vdots\\
	c_{1m}K_1+ c_{2m}K_2+ \ldots + c_{rm}K_r
	\end{array}
	\right]\\
&=	 \left[\begin{array}{cccc}
	c_{11}&c_{21}& \ldots &c_{r1}\\
	c_{12}& c_{22}& \ldots & c_{r2}\\
	\vdots\\
	c_{1m}& c_{2m}& \ldots &c_{rm}
	\end{array}
	\right]  \left[\begin{array}{c}
	K_1\\
	K_2\\
	\vdots\\
	K_r
	\end{array}
	\right]\\
&=NK.
\end{align*}

\medskip

Thus, $f^*=N^*K^*=NK=f$ and so $(\mathscr{N}^*, K^*)$ is dynamically equivalent to $(\mathscr{N},K)$ under STAR transformation.

One particular idea of STAR transformation is to use the maximal stoichiometric coefficient (MSC) among the complexes in the CRN to construct reactions whose reactant complexes and product complexes are different from existing ones. This is done by uniform translation of the reactants and products per linkage class. The method creates $h$ replicas of each linkage class, and hence $\mathscr{N}^*_M$ becomes the union.
 
We now describe \textbf{S-invariant Termwise Addition of Reactions Via Maximal Stoichiometric Coefficients (STAR-MSC)}  in detail. Note that the definition domain of a poly-PL kinetics is $\mathbb{R}^{\mathscr{S}}_{> 0}$. Hence, all $x=(x_1,\ldots,x_m)$ are positive vectors. Let $M = 1 + \max \{y_i | y \in \mathscr{C}\}$, where the second summand is the maximal stoichiometric coefficient (a positive integer). Admitting redundancy in notations, for every positive integer $z$, let $z$ be identified with the vector $(z,z,\ldots,z)$ in $\mathbb{R}^{\mathscr{S}}$.  Let $\mathscr{L}_1,\ldots,\mathscr{L}_\ell$ be the linkage classes of $\mathscr{N}$.  
 
For each complex $y \in \mathscr{C}$, form the $(h-1)$ complexes 
\begin{displaymath}
y + M, y + 2M, \ldots, y + (h -1)M.  
\end{displaymath}

If $y$ and $y'$ are different complexes but $\ds y + jM = y'+ j'M$ and say, $j' \geq j$, we have $y - y'= (j'- j)M$. Note that we have a contradiction, since the RHS of this equation is either the zero vector or a positive one with all coefficients greater than $M - 1$ and on the LHS is a  vector with coefficients less than $M - 1$.
This implies that for each reaction $r_i \in \mathscr{R}$, we also obtain $(h -1)$ reactions different from each other and all existing reactions. 

We denote these reactions with $r_{ij}$. In this way, we obtain $(h-1)$ different replicas of $\mathscr{N}$  which we denote by $\mathscr{N}_2,\ldots  \mathscr{N}_h$. For convenience, we set $\mathscr{N}_1 = \mathscr{N}$.  We introduce STAR-MSC transfrom as $\mathscr{N}^*_M= \bigcup \mathscr{N}_i$.

For a reaction $r_{ij}$ in $\mathscr{N}_j$ , $j = 1,\ldots,h$, we now define the rate function $K^*_{M\;ij} (x) = k_i a_{ij}M_{ij}$. 

\begin{example}
\upshape{
For $(\mathscr{N},K)$,  suppose $\mathscr{S} =\{X,Y\}$, $\mathscr{C}=\{X, 2X, 2X+3Y,X+2Y\}$ and  $\mathscr{R}=\{r_1:X\rightarrow2X,r_2:2X+3Y\rightarrow X+2Y\}$. Let the poly-PL kinetics be given by:
\begin{displaymath}
K_{r_1} (X,Y)=k_1 (5X^2Y^3+Y^2 )\; \hspace{1cm}
K_{r_2} (X,Y)=k_2 (0.3X^2 Y+0.4X^3) 
\end{displaymath}
where $k_1$ and $k_2$ are usual rate constants.

As observed, the maximal stoichiometric coefficient is 3. Hence, $M=4X+4Y$. With this, we now generate the STAR-MSC transform $(\mathscr{N}^*_M,K^*_M)$ where $\mathscr{S}=\{X,Y\}$. 

The additional complexes are $\{5X+4Y, 6X+4Y, 6X+7Y, 5X+6Y\}$ and  the additional reactions are $\{r_1^*: 5X+4Y \rightarrow 6X+4Y, r_2^*: 6X+7Y \rightarrow 5X+6Y\}$. This means that $R^*_M=\{r_1:X\rightarrow2X, r_1^*: 5X+4Y \rightarrow 6X+4Y, r_2:2X+3Y\rightarrow X+2Y, r_2^*: 6X+7Y \rightarrow 5X+6Y\}$ with the corresponding kinetic functions:

\begin{displaymath}
K_{r_1}^* (X,Y)= 5k_1 X^2Y^3=k_1^* X^2Y^3 \; \hspace{1cm} K_{r_1^*}^* (X,Y)= k_1Y^2=k_1^{**} Y^2
\end{displaymath}
\begin{displaymath}
K_{r_2}^* (X,Y)=0.3k_2 X^2 Y=k_2^* X^2 Y \; \hspace{1cm}
K_{r_2^*}^* (X,Y)= 0.4k_2 X^3=k_2^{**} X^3.
\end{displaymath}
}
\end{example}

Throughout this paper, we use the following as our running example:

\begin{example}
\upshape{\textit{(Running Example):}
For $(\mathscr{N},K)$,  suppose $\mathscr{S} =\{X,Y\}$, $\mathscr{C}=\{5X+Y, X+3Y\}$ and  $\mathscr{R}=\{r_1:5X+Y\rightarrow X+3Y,r_2:X+3Y\rightarrow 5X+Y\}$. Let the poly-PL kinetics be given by the following (with $k_1$ and $k_2$ are the rate constants as usual).
\[
K_{r_1} (X,Y)=k_1 (\alpha_1 X^2Y+  \alpha_2 XY^2 + \alpha_3 X^2Y^2 +\alpha_4 X^3Y )
\]
\[
K_{r_2} (X,Y)=k_2 (\beta_1 X+  \beta_2 Y + \beta_3 XY + \beta_4 X^2) 
\]

Since the maximal stoichiometric coefficient is 5, $M=6X+6Y$. Hence, we generate STAR-MSC transform $(\mathscr{N}^*_M,K^*_M)$ where $\mathscr{S}=\{X,Y\}$. The following will be our subsetworks:
\begin{enumerate}[i.]
\item $\mathscr{N}_1=\mathscr{N}: \mathscr{C}_1= \{5X+Y, X+3Y\}$ and $\mathscr{R}_1=\{r_1:5X+Y\rightarrow X+3Y,r_2:X+3Y\rightarrow 5X+Y\}$;
\item $\mathscr{N}_2: \mathscr{C}_2= \{11X+7Y, 7X+9Y\}$ and  $\mathscr{R}_2=\{r_1^*:11X+7Y\rightarrow 7X+9Y,r_2^*:7X+9Y\rightarrow 11X+7Y\}$;
\item $\mathscr{N}_3: \mathscr{C}_3= \{17X+13Y, 13X+15Y\}$ and  $\mathscr{R}_3=\{r_1^{**}:17X+13Y\rightarrow 13X+15Y,r_2^{**}:13X+15Y\rightarrow 17X+13Y\}$
\item $\mathscr{N}_4: \mathscr{C}_4= \{23X+19Y, 19X+21Y\}$ and  $\mathscr{R}_4=\{r_1^{***}:23X+19Y\rightarrow 19X+21Y,r_2^{***}:19X+21Y\rightarrow 23X+19Y\}$.
\end{enumerate}

With this, $\ds \mathscr{N}^*_M= \bigcup_{i=1}^4 \mathscr{N}_i$. The following are the corresponding kinetic functions:
\begin{align*}
K_{r_1}^* (X,Y)&=k_1^* X^2Y \qquad \quad \ \  K_{r_1^*}^* (X,Y)=k_1^{**} XY^2\\
K_{r_1^{**}}^* (X,Y)&=k_1^{***} X^2Y^2 \qquad K_{r_1^{***}}^* (X,Y)=k_1^{****} X^3Y\\
K_{r_2}^* (X,Y)&=k_2^* X \qquad \qquad \ \ \ K_{r_2^*}^* (X,Y)=k_2^{**} Y\\
K_{r_2^{**}}^* (X,Y)&=k_2^{***} XY \qquad \ \  K_{r_2^{***}}^* (X,Y)=k_2^{****} X^2.
\end{align*}
}
\end{example}

With regard to the network numbers, Table \ref{table:networknumbers} shows the comparison between the original network and the STAR-MSC transform.

\begin{table}[ht]
\caption{Network Numbers of STAR-MSC Transform}
\label{table:networknumbers}
\centering
\begin{tabular}{p{5.5cm}l}
& \\
\hline
\textbf{Network Number} & \textbf{Value/Bounds} \\ \hline
Number of species & $m^*_M = m$ \\ \hline
Number of complexes & $n^*_M= hn$ \\ \hline
Number of reactant complexes
& $n^*_{r\ M} = hn_r$ \\ \hline
Number of reactions & $r^*_M= hr$ \\ \hline
Number of linkage classes & $\ell^*_M= h\ell$ \\ \hline
Number of strong linkage & \multirow{2}{*}{$(s\ell)^*_M= h(s\ell)$}\\
classes & \\ \hline
Number of terminal strong & \multirow{2}{*}{$t^*_M= ht$}\\
linkage classes &  \\ \hline
Rank of network & $s^*_M= s$ \\ \hline
Reactant rank of the network & $q^*_M= \left\{
        \begin{array}{ll}
            q & if\;\;(1,1,\ldots,1) \in R \\
            q+1 & if\;\;(1,1,\ldots,1) \notin R
        \end{array}
    \right.$ \\ \hline
Rank difference & $\Delta(\mathscr{N})^*_M= \left\{
        \begin{array}{ll}
            \Delta(\mathscr{N}) & if\;\;(1,1,\ldots,1) \in R\\
            \Delta(\mathscr{N})-1 & if\;\;(1,1,\ldots,1) \notin R
        \end{array}
    \right.$ \\ \hline
Deficiency of the network & $\delta^*_M= \delta + (h-1) (n-\ell)$ \\ \hline
Reactant deficiency of the network & $\delta^*_{\rho\;M} \geq \delta_\rho$ \\ 
& $\delta^*_{\rho\;M} = \left\{
        \begin{array}{ll}
            hn_r-q & if\;\;(1,1,\ldots,1) \in R \\
            hn_r-q-1 & if\;\;(1,1,\ldots,1) \notin R
        \end{array}
    \right.$ \\ \hline
\end{tabular}
\end{table}

Note that when we generate the STAR-MSC transform, no species were added and so $m^* = m$.

On the other hand, the number of complexes, reactant complexes, reactions, linkage classes, strong linkage classes and terminal strong linkage classes are all dependent on results of the following  facts:  (a) all new complexes are unique and different from all existing complexes; (b) all new reactions are unique and different from all existing reactions; and (c) $\mathscr{N}^*_M$ is composed of $h$ replicas of $\mathscr{N}$.

For the rank of the network, observe that the additional reaction vectors are of the form $(y'+ jM)-(y+jM)=y'-y$. With this, all additional reaction vectors are part of $S$. Hence, $S^*=S$ and so $s^*=s$.

Note that an arbitrary $M$ in the process of STAR-MSC is an element of $\langle(1,1, \ldots,1) \rangle$. This means that, based on the construction of additional complexes,  the reaction subspace of the STAR-MSC transform is $R^*_M= \langle \rho(\mathscr{R}) \cup \{(1,1, \ldots, 1)\}\rangle$. This is the basis of the values of reactant rank, rank difference and reactant deficiency of the tranformed network.

As with the deficiency of the network, it is clear that 
\begin{displaymath}
\delta^*_M =n^*_M-\ell^*_M-s^*_M= hn-h\ell-s = (h-1)(n-\ell)+n-\ell-s=\delta+ (h-1)(n-\ell).
\end{displaymath}

Since the $\mathscr{N}^*_1 = \mathscr{N}$ and the $(h -1)$ $\mathscr{N}$-replicas $\mathscr{N}^*_2, \ldots, \mathscr{N}^*_h$ form a $\mathscr{C}$-decomposition of $\mathscr{N}^*$, most network properties of $\mathscr{N}$ are also valid for $\mathscr{N}^*$. These include frequently required properties such as:

\begin{enumerate}[i)]
\item (weak) reversibility,
\item $t$-minimality,
\item cycle (point) terminality,
\item terminality bounded by deficiency (TBD),
\item sufficient reactant diversity (SRD), and
\item stoichiometric subspace containment in the reactant subspace (SRS).
\end{enumerate}

 There are however some network properties which may not be preserved under STAR-MSC. The property TND may not be invariant as the following equation show:
\begin{displaymath}
t^*-\ell^*-\delta^*=ht-h\ell-\delta+(h-1)(n-\ell)=(h-1)(t-n)+(t-\ell-\delta).
\end{displaymath}

The same holds for LRD.  Since the reactant rank is in general not determined by the numbers of $\mathscr{N}$, it is also a source of variance. In general, $R^*$ is a superset of $R$, so that $im\; Y^* = R^* + S^* = R^* + S$ may vary too.  The properties RSS and RES may be affected too. For a systematic analysis of the relationships between network and subnetwork properties under decompositions, see \cite{FOMF2020}.

Since STAR-MSC is a dynamic equivalence, all purely kinetic properties are also maintained. However, some key structo-kinetic properties are not, e.g., if on a cycle terminal network, the poly-PL kinetics is factor span surjective (i.e., $span\; (im\; \psi K) = R)^{\mathscr{C}}$, where $\psi K$ is the factor map of the PY-RDK kinetics $K$), the PL-RDK kinetics $K^*$ need not be factor span surjective.

A very important relationship holds between the sets of complex balanced equilibria of $\mathscr{N}^*$ and those of the $\mathscr{N}^*_j$ as shown in \cite{FAML2020}: ${Z}_+(\mathscr{N}^*,K^*)={Z}_+(\mathscr{N}^*_1,K^*_1) \cap \ldots \cap {Z}_+(\mathscr{N}^*_h,K^*_h)$. In fact, the following equivalence holds: ${Z}_+(\mathscr{N}^*,K^*)\neq \phi \Leftrightarrow {Z}_+(\mathscr{N}^*_j,K^*_j)\neq \phi$ for $j=1, \ldots, h$.

\section{Application of the Multistationarity Algorithm\\ (MSA) to a STAR-MSC Transform}
\label{application:MSA:STAR}
In this section, we present a review of the MSA for PLK systems and how we can further extend the algorithm to determine the multistationarity of PYK kinetic systems via the STAR-MSC transformation.

\subsection{A review of the MSA for PLK systems}

In 2011, Ji and Feinberg \cite{ji} introduced the higher deficiency algorithm (HDA). It is a general method of solving whether a chemical reaction network (CRN) endowed with mass action kinetics (MAK) has the capacity to admit multiple equilibria. This was extended by Hernandez et al. \cite{hmr2019} in 2020 for CRNs endowed with power law kinetics, which is a superset of MAK, and was called the Multistationarity Algorithm (MSA) for power law kinetic (PLK) systems.
The MSA combines the extension of the HDA and CF transformations.
If we have a PL-RDK system, then the extended HDA can be directly applied. For a PL-NDK system, the CF-RM method in \cite{neml2019} is used to transform the system to a ``dynamically equivalent'' PL-RDK system \cite{hmr2019}.

The fundamental decomposition of a CRN is the set of subnetworks generated by the partition of its reaction set into ``fundamental classes'', which is actually the basis of the HDA. A CRN has independent decomposition if the network's stoichiometric subspace is the direct sum of the subnetworks' stoichiometric subspaces.
It was shown in \cite{hmr22019} that for any CRN with independent fundamental decomposition, the CF-RM transformation, required in the MSA, is not necessary. Hence, we can directly apply the extended HDA for RDK or NDK systems with independent fundamental decompositions. Otherwise, the CF-RM transformation is used to convert the system to a PL-RDK system before applying the extended HDA. A MATLAB program which determines whether a CRN has independent fundamental decomposition is provided in \cite{bsh2020}.

The main contribution of this paper is to provide a general method of solving the problem of multistationarity for a much larger class of kinetic systems: CRNs endowed with poly-PL kinetics, i.e., kinetics which are non-negative linear combinations of power law functions.

For the detailed process and the steps of the MSA, the reader may refer to \cite{ji, {hmr2019}}. We now introduce the concepts of orientations and fundamental classes for the preliminary steps of the MSA. An {\bf orientation} $\mathscr{O}$ is a subset of $\mathscr{R}$ such that for every reaction $y \to y' \in \mathscr{R}$, either $y \to y' \in \mathscr{O}$ or $y' \to y \in \mathscr{O}$, but not both. For a given orientation $\mathscr{O}$, we define a linear map ${L_\mathscr{O}}:{\mathbb {R}^\mathscr{O}} \to S$ such that
\[{L_\mathscr{O}}(\alpha)= \sum\limits_{y \to y' \in \mathscr{O}} {{\alpha _{y \to y'}}\left( {y' - y} \right)}.\]
Now, let $\left\{ {{v^l}} \right\}_{l = 1}^d$ be a basis for ${\text{Ker }}{L_\mathscr{O}}$.
If for $y \to y' \in \mathscr{O}$, $v_{y \to y'}^l =0$ for all $1 \le l \le d$ then the reaction $y \to y' $ belongs to the zeroth equivalence class $P_0$.
For $y \to y', {\overline y  \to \overline y '} \in \mathscr{O} \backslash P_0$, if there exists $\alpha \ne 0$ such that $v_{y \to y'}^l = \alpha v_{\overline y  \to \overline y '}^l$ for all $1 \le l \le d$, then the two reactions are in the same {\bf equivalence class} denoted by $P_i$, $i \ne 0$.
On the other hand, the reactions $y \to y'$ and $\overline y  \to \overline y '$ in $\mathscr{R}$ belong to the same {\bf fundamental class} if at least one of the following is satisfied:
\begin{itemize}
\item[i.] $y \to y'$ and $\overline y  \to \overline y '$ are the same reaction.
\item[ii.] $y \to y'$ and $\overline y  \to \overline y '$ are reversible pair.
\item[iii.] Either $y \to y'$ or $y' \to y$, and either $\overline y  \to \overline y '$ or $\overline y'  \to \overline y $ are in the same equivalence class on $\mathscr{O}$.
\end{itemize}
Hence, the orientation $\mathscr{O}$ is partitioned into equivalence classes while the reaction set $\mathscr{R}$ is partitioned into fundamental classes.

We now briefly describe the steps of the MSA.

\noindent {\bf STEP 1}: Choose an orientation.

\noindent {\bf STEP 2}: Partition the orientation into equivalence classes and the reaction set into fundamental classes. If at least one of the following conditions is not satisfied, we conclude that the system does not have the capacity to admit multiple equilibria:
\begin{itemize}
\item[i.] All reactions in $P_0$ are reversible (with respect to $\mathscr{R}$).
\item[ii.] For two irreversible reactions (with respect to $\mathscr{R}$), $y \to y'$ and ${\overline y  \to \overline y '}$ in the same $P_i$, there exists $\alpha > 0$ such that $v_{y \to y'}^l = \alpha v_{\overline y  \to \overline y '}^l$ for all $1 \le l \le d$.
\end{itemize}
\noindent {\bf STEP 3}: Find the colinkage sets. Each fundamental class corresponds to a subnetwork of the whole network. For each subnetwork, we identify terminal and nonterminal strong linkage classes.

\noindent {\bf STEP 4}: Pick $W \subseteq {\mathscr{O}}$.
An equivalence class is reversible if all of its reactions are reversible with respect to the original network $\mathscr{R}$. It is nonreversible, if it contains an irreversible reaction.
In this step, we pick a representative reaction for each of the $P_i$'s such that:
If a class $P_i$ is nonreversible, we pick an irreversible reaction. Otherwise, we pick any reversible reaction.

\noindent {\bf STEP 5}: Realign the orientation, if needed.
If the following statement is not satisfied, we choose another orientation until it is satisfied.
For $P_i$ with $1 \le i \le w$, for any reaction ${y \to y'}$ in $P_i$, there exists an $\alpha_{y \to y'} >0$ such that $v_{y_i \to y'_i}^l = \alpha_{y \to y'} v_{y  \to  y '}^l$ for all the elements $v_1, v_2,...,v_d$ of the chosen basis for ${\text{Ker }}L_{\mathscr{O}}$.

\noindent {\bf STEP 6}: Find a basis for ${\text{Ker}}^{\perp}{\text{ }}L_{\mathscr{O}}\cap \Gamma_W$ where
${\Gamma _W} = \left\{ {x \in {\mathbb{R}^\mathscr{O}}|x{\rm{ \ has \ support \ in \ }}W} \right\}$ and $W = \left\{ {{y_i} \to {y_i}'|i = 1,...,w} \right\} \subseteq \mathscr{O}$.

\noindent {\bf STEP 7}: Check the linearity of the system of inequalities.
If there is a ``forest basis'' (see \cite{{hmr2019},ji}) for ${\text{Ker}}^{\perp}{\text{ }}L_{\mathscr{O}}\cap \Gamma_W$, then the resulting inequality system is linear. Otherwise, some nonlinear constraints might be needed and hence the system may be nonlinear.

\noindent {\bf STEP 8}: Choose signs for vectors ${g_W}=g|_W,{h_W}=h|_W \in  {\mathbb{R}^\mathscr{O}} \cap \Gamma _W$.
We choose two sign patterns for ${\mathbb{R}}^{\mathscr{O}}$ such that the following hold:
\begin{itemize}
\item[i.] The sign patterns are not the zero vector at the same time.
\item[ii.] Both of the sign patterns are sign-compatible (the signs of corresponding components agree with each other) with ${\text{Ker }}L_{\mathscr{O}}$.
\item[iii.] For each sign pattern and each nonreversible equivalence class $P_i$, the sign assigned on the representative (and all reactions in $P_i$) is positive.
\end{itemize}

We need to introduce the following definitions before we proceed with the next step.

\begin{definition} {\bf \cite{muller}}
\label{tmatrix}
The $m \times n$ matrix $\widetilde Y$ is given by 
\[{\left( {\widetilde Y} \right)_{ij}} = \left\{ \begin{array}{cl}
{\left( F \right)_{ki}} & {{\rm{   if}} \ j \ {\rm{ is \ a \ reactant \ complex \ of \ reaction \ }} k},\\
0  &     {\rm  otherwise}.
\end{array} \right.\]
\end{definition}

\begin{definition}
The $T$-matrix is the $m \times n_r$ truncated $\widetilde Y$ matrix where the nonreactant columns are removed.
\end{definition}

For the succeeding discussions, for a reaction ${y\to y'}$, $T_{.y}$ refers to the column of the $T$-matrix associated with the reactant complex $y$.

\noindent {\bf STEP 9}: Put reactions in the nondegenerate $C_i$'s into shelves.
A fundamental class $C_i$ with $0 \le i \le w$ is called degenerate if $g_{y_i\to y_{i}'}=0$ while a fundamental class $C_i$ with $1 \le i \le w$ is called nondegenerate if $g_{y_i\to y_{i}'}\ne 0$. For each nondegenerate fundamental class $C_i$ with $i \ge 1$, assume a 3-shelf bookcase to store all reactions in $C_i$. Let $y \to y'$ be a reaction in a nondegenerate fundamental class $C_i$, and 
${\rho _{{y_i} \to {y_i}'}} = \dfrac{{{h_{{y_i} \to {y_i}'}}}}{{{g_{{y_i} \to {y_i}'}}}}$ where ${g_{{y_i} \to {y_i}'}} \ne 0$ and $i = 1,...,w$. Define the shelving of the reaction $y \to y'$ using the following statements:
\begin{itemize}
\item[i.] $y \to y'$ is on the upper shelf if ${e^{{T_{.y}} \cdot \mu }} > \rho_{y_i\to y_{i}'}$.
\item[ii.] $y \to y'$ is on the lower shelf if ${e^{{T_{.y}} \cdot \mu }} < \rho_{y_i\to y_{i}'}$.
\item[iii.] $y \to y'$ is on the middle shelf if ${e^{{T_{.y}} \cdot \mu }} = \rho_{y_i\to y_{i}'}$.
\end{itemize}
We let ${M_{{y_i} \to {y_i}'}} = \ln {\rho _{{y_i} \to {y_i}'}}$ if ${\rho _{{y_i} \to {y_i}'}}>0$ and otherwise, take ${M_{{y_i} \to {y_i}'}}$ to be an arbitrarily large and negative number. For this step, we refer to the following lemma.
\begin{lemma} Let $\left(\mathscr{S},\mathscr{C},\mathscr{R},K\right)$ be a PL-RDK system and $\mathscr{O}$ be an orientation. Suppose there exist $\mu  \in {\mathbb{R}^\mathscr{S}}$, $g,h \in {\text{Ker }}L_\mathscr{O}$, ${P_i} \ (i=1,2,...,w)$ with representative ${{y_i} \to {y_i}'}$
and $\left\{ {{\rho _{{y_i} \to {y_i}'}} = \dfrac{{{h_{{y_i} \to {y_i}'}}}}{{{g_{{y_i} \to {y_i}'}}}}|{g_{{y_i} \to {y_i}'}} \ne 0,i = 1,...,w} \right\}$ satisfy the conditions given in Lemma \ref{rev_irrev4}. Then the following holds for a nondegenerate $C_i$:
  \begin{itemize}
\item[i.] All irreversible reactions in $C_i \ (i\ge 1)$ must belong to the middle shelf.
\item[ii.] $y \to y' \in {C_i}$ must belong to the upper shelf if ${{\rho _{{y_i} \to {y_i}'}}}\le 0$.
\item[iii.] If a reaction is reversible, then the reaction and its reversible pair must belong to the same shelf.
\item[iv.] Any two reactions in $C_i$ with the same reactant complex must belong to the same shelf.
\item[v.] Each reaction whose reactant complex lies in a nonterminal strong linkage class must belong to the middle shelf.
\item[vi.] Each reaction whose reactant complex lies in a terminal strong linkage class of the fundamental subnetwork must belong to the same shelf.
\item[vii.] If for a nondegenerate $C_i \ (i\ge 1)$, $\mathscr{N}_i$ forms a big (undirected) cycle (with at least three vertices), then its reactions are all in a terminal strong linkage class and belong to the middle shelf, where $\mathscr{N}_i$ is the subnetwork generated by reactions in $C_i$.
  \end{itemize}
\label{rev_irrev2s}
\end{lemma}

\noindent For STEPS 10, 11, and 12, we refer to the conclusion of the following lemma.
\begin{lemma} \cite{{hmr2019}} Let $\left(\mathscr{S},\mathscr{C},\mathscr{R},K\right)$ be a PL-RDK system and $\mathscr{O}$ be an orientation. Let $\kappa  \in \mathbb{R}_{ > 0}^\mathscr{R}$, and $\mu  \in {\mathbb{R}^\mathscr{S}}$. Let $g,h \in \mathbb {R}^\mathscr{O}$ such that\\
${g_{y \to y'}} = \left\{ \begin{array}{ll}
{\kappa _{y \to y'}} - {\kappa _{y' \to y}}&{\rm{   if }} \ y \to y' \in \mathscr{O}{\rm{\  is \ reversible}}\\
{\kappa _{y \to y'}}&{\rm{          if }} \ y \to y' \in \mathscr{O}{\rm{ \  is \  irreversible}}
\end{array} \right.$
and\\
${h_{y \to y'}} = \left\{ \begin{array}{ll}
{\kappa _{y \to y'}}{e^{{T_{.y}} \cdot \mu }} - {\kappa _{y' \to y}}{e^{{T_{.y'}} \cdot \mu }}&{\rm{  if }} \ y \to y' \in \mathscr{O}{\rm{\ is \ reversible}}\\
{\kappa _{y \to y'}}{e^{{T_{.y}} \cdot \mu }}&{\rm{           if }} \ y \to y' \in \mathscr{O}{\rm{\ is \ irreversible}}
\end{array} \right..$\\
For $i=1,2,...,w$, let ${P_i}$ be the equivalence class with ${y_i} \to {y_i}'$ as representative. Moreover, let ${{\rho _{{y_i} \to {y_i}'}} = \dfrac{{{h_{{y_i} \to {y_i}'}}}}{{{g_{{y_i} \to {y_i}'}}}}}$ for nondegenerate fundamental classes ${C_i}$.
\begin{itemize}
\item[i.] If $y \to y' \in {P_i}\ (i=1,2,...,w)$ is irreversible then ${g_{y \to y'}} > 0$, ${h_{y \to y'}} > 0$, and ${M_{{y_i} \to {y_i}'}} = {{{T_{.y}} \cdot \mu }}$.
\item[ii.] Suppose $y \to y' \in {P_i} \ (i=1,2,...,w)$ is reversible and ${C_i}$ is nondegenerate.
\begin{itemize}
\item[a.] If ${g_{y_i \to y_i'}} > 0$ and $y \to y'$ is on the upper shelf, then ${M_{{y_i} \to {y_i}'}} < {{{T_{.y}} \cdot \mu }} < {{{T_{.y'}} \cdot \mu }}$.
\item[b.] If ${g_{y_i \to y_i'}} > 0$ and $y \to y'$ is on the middle shelf, then ${M_{{y_i} \to {y_i}'}} = {{{T_{.y}} \cdot \mu }} = {{{T_{.y'}} \cdot \mu }}$.
\item[c.] If ${g_{y_i \to y_i'}} > 0$ and $y \to y'$ is on the lower shelf, then ${M_{{y_i} \to {y_i}'}} > {{{T_{.y}} \cdot \mu }} > {{{T_{.y'}} \cdot \mu }}$.
\item[d.] If ${g_{y_i \to y_i'}} < 0$ and $y \to y'$ is on the upper shelf, then ${M_{{y_i} \to {y_i}'}} < {{{T_{.y'}} \cdot \mu }} < {{{T_{.y}} \cdot \mu }}$.
\item[e.] If ${g_{y_i \to y_i'}} < 0$ and $y \to y'$ is on the middle shelf, then ${M_{{y_i} \to {y_i}'}} = {{{T_{.y'}} \cdot \mu }} = {{{T_{.y}} \cdot \mu }}$.
\item[f.] If ${g_{y_i \to y_i'}} < 0$, and $y \to y'$ is on the lower shelf, then ${M_{{y_i} \to {y_i}'}} > {{{T_{.y'}} \cdot \mu }} > {{{T_{.y}} \cdot \mu }}$.
\end{itemize}
\item[iii.] Suppose $y \to y' \in {P_i}\ (i=1,2,...,w)$ is reversible and ${C_i}$ is degenerate.
\begin{itemize}
\item[a.] If ${h_{y_i \to y_i'}} > 0$ then ${{{T_{.y}} \cdot \mu }} > {{{T_{.y'}} \cdot \mu }}$.
\item[b.] If ${h_{y_i \to y_i'}} = 0$ then ${{{T_{.y}} \cdot \mu }} = {{{T_{.y'}} \cdot \mu }}$.
\item[c.] If ${h_{y_i \to y_i'}} < 0$ then ${{{T_{.y}} \cdot \mu }} < {{{T_{.y'}} \cdot \mu }}$.
\end{itemize}
\item[iv.] If $y \to y' \in {P_0}$ is reversible then ${{{T_{.y}} \cdot \mu }} = {{{T_{.y'}} \cdot \mu }}$.
\end{itemize}
\label{rev_irrev4}
\end{lemma}
\noindent To provide an easier notation, we use $M_i$ instead of ${M_{{y_i} \to {y_i}'}}$.

\noindent {\bf STEP 10}: Put shelving equalities and inequalities from the nondegenerate $C_i$'s. 

\noindent {\bf STEP 11}: Obtain additional upper and lower shelving inequalities from $P_i$'s with nondegenerate $C_i$'s.

\noindent {\bf STEP 12}: Obtain additional equalities and inequalities from $P_i$'s with degenerate $C_i$'s. 

\noindent {\bf STEP 13}: Add $M$ equalities and inequalities. We refer to the following statement.\\
Two multisets $Q_1$ and $Q_2$ are nonsegregated if at least one of the following two cases holds:
\begin{itemize}
\item[i.] There exist $a$ from $Q_1$ or $Q_2$, and $b<c$ from the other such that $b<a<c$.
\item[ii.] All elements in the two multisets are equal, or there exist $a,b\in Q_1$ and $c,d\in Q_2$ such that $c=a<b=d$.
\end{itemize}
For the full details, we refer the reader to \cite{ji}.\\
\noindent {\bf STEP 14}: Check for solution(s) of the system of inequalities.
If the system is linear, then the system obtained from STEP 10 to STEP 13 is complete. The solution (vector $\mu$), which must be sign-compatible with an element of the stoichiometric subspace, will be called a {\bf{signature}}, if it exists, and the kinetic system has the capacity to admit multiple equilibria. If a solution does not exist, STEPS 8-14 may be repeated. If for all possible systems of inequalities, solutions do not exist, then the kinetic system does not have the capacity to admit multiple equilibria within the stoichiometric class.

\subsection{Implementation of the extension for STAR-MSC transform of a PYK kinetic system}
A positive equilibrium of a PYK system $\left ( \mathscr{N},K \right )$ is also a positive equilibrium of its transform $\left ( \mathscr{N}^*,K^* \right )$ for the set of rate constants given by the original PYK system. For $j=1,2,...,h$, we form $\mathscr{R}^*_j=\{r_{ij}|i=1,2,...,r\}$ which induces a partition of $\mathscr{R}^*$ generating the decomposition $\mathscr{N}^*=\mathscr{N}_1 \cup \mathscr{N}_2 \cup ... \cup \mathscr{N}_h$. Note that $h-1$ different replicas of $\mathscr{N} = \mathscr{N}_1$ are obtained denoted by $\mathscr{N}_2$, $\mathscr{N}_3$, ... , $\mathscr{N}_{h}$. Together with the restriction of the rate function $K^*$, they form PLK subsystems such that for reaction $r_{ij}$ in $N_j$ for each $j$, we have $K^*_{ij}(x)=k_ia_{ij}M_{ij}$. After applying the MSA to the transform $\left ( \mathscr{N}^*,K^* \right )$, we have the following statements.
\begin{itemize}
\item[i.] If there are no rate constants for which $\left ( \mathscr{N}^*,K^* \right )$ has the capacity for multistationarity (i.e., it is monostationary for all stoichiometric classes), then the original system $\left ( \mathscr{N},K \right )$ is also monostationary for all stoichiometric classes.
\item[ii.] Otherwise, if after solving for the possible multiple equilibria in the MSA for the PLK system and assuming these equilibria for the original PYK sytem, there exist corresponding rate constants, then the original system admits multiple equilibria in a stoichiometric class.
\end{itemize}

\subsection{Computations from application to STAR-MSC transform of the running example}
The STAR-MSC transform in the running example has the following reaction network together with the corresponding kinetics.
\[ \begin{array}{llll}
{\text {Reactions}} & {\text {Kinetics}}  & {\text {Reactions}}   & {\text {Kinetics}}\\
R_1: 5X+Y \to X+3Y & \alpha_1k_1X^2Y &  R_2: X+3Y  \to 5X+Y & \beta_1k_2X\\
R_3: 11X+7Y \to 7X+9Y  &  \alpha_2k_1XY^2 & R_4: 7X+9Y \to 11X+7Y & \beta_2k_2Y\\
R_5: 17X+13Y \to 13X+15Y &  \alpha_3k_1X^2Y^2 & R_{6}: 13X+15Y \to 17X+13Y & \beta_3k_2XY\\
R_7: 23X+19Y \to 19X+21Y  &  \alpha_4k_1X^3Y & R_{8}: 19X+21Y \to 23X+19Y  & \beta_4k_2X^2\\
\end{array}\]
Hence, the transpose of the kinetic order matrix is
\[\bordermatrix{%
& R_1 & R_2 & R_3 & R_4 & R_5 & R_6 & R_7 & R_8 \cr
X &2&1&1&0&2&1&3&2  \cr
Y &1&0&2&1&2&1&1&0  \cr
}.\]

We choose the following orientation: ${\mathscr{O}}=\{R_1,R_3,R_6,R_8\}.$
Below is basis for ${\text{Ker }}{L_\mathscr{O}}$ obtained by solving $\sum\limits_{y \to y' \in \mathscr{O}} {{\alpha _{y \to y'}}\left( {y' - y} \right)}=0$.
\[\bordermatrix{
& v^1 & v^2 & v^3\cr
R_1 & -1 & 1 & 1  \cr
R_3 & 1 & 0 & 0  \cr
R_6 & 0 & 1 & 0  \cr
R_8 & 0 & 0 & 1  \cr
}\]
By inspecting the rows of the matrix, we obtain four equivalence classes: $\{R_1\}$, $\{R_3\}$, $\{R_6\}$, and $\{R_8\}$. Since reversible pairs belong to the same fundamental class, we obtain the following fundamental classes: $\{R_1,R_2\}$, $\{R_3,R_4\}$, $\{R_5,R_6\}$, and $\{R_7,R_8\}$.
In addition, a basis for ${\text{Ker}}^{\perp}{\text{ }}L_{\mathscr{O}}\cap \Gamma_W$ is 
$\{{b_1} = \left( {\begin{array}{*{20}{c}}
{ - 1}&{ - 1}&1&1
\end{array}} \right)\}$.
We pick the sign patterns to be positive. In addition, we choose the case when each of the reactions belongs to the middle shelf.\\
${{\cal M}_1} = \left\{ { R_1: 5X+Y \to X+3Y,  R_2: X+3Y  \to 5X+Y   } \right\}$\\
${{\cal M}_2} = \left\{ { R_3: 11X+7Y \to 7X+9Y,  R_4: 7X+9Y \to 11X+7Y  } \right\}$\\
${{\cal M}_3} = \left\{ {R_5: 17X+13Y \to 13X+15Y, R_{6}: 13X+15Y \to 17X+13Y
} \right\}$\\
${{\cal M}_4} = \left\{ {  R_7: 23X+19Y \to 19X+21Y, R_{8}: 19X+21Y \to 23X+19Y } \right\}$

We choose $M_1 > M_3 > M_2$ from $b^1$. Thus, we obtain the following system.
\[ \begin{array}{c}
2{\mu _{{X}}} +{\mu _{{Y}}} = {\mu _{{X}}}={M_1}\\
{\mu _{{X}}} +2{\mu _{{Y}}} = {\mu _{{Y}}}={M_2}\\
2{\mu _{{X}}} +2{\mu _{{Y}}} = {\mu _{{X}}} +{\mu _{{Y}}}={M_3}\\
3{\mu _{{X}}} +{\mu _{{Y}}} = 2{\mu _{{X}}} ={M_4}\\
M_1 > M_3 > M_2
\end{array}\]
The given system has the solution $(\mu_X,\mu_Y)=(1,-1)$. A basis for the stoichiometric subspace is $\{-4X+2Y\}$. Hence, we can choose $\sigma=(2,-1)$ which is of course sign-compatible with $\mu=(1,-1)$. We can compute these possible equilibria: $c^{*}=(3.163953414,0.581976707
)$ and $c^{**}=(1.163953414,1.581976707)$.

If the kinetic system has indeed two positive and distinct equilibria, $c^{*}$ and $c^{**}$, then 
$$
\sum\limits_{y \to y' \in {\mathscr{R}}} {{\kappa _{y \to y'}}\left( {y' - y} \right) = 0} 
\label{equi3}
{\rm{ \ and \ }}
\sum\limits_{y \to y' \in {\mathscr{R}}} {{\kappa _{y \to y'}}{e^{{T_{.y }\cdot \mu}}}\left( {y' - y} \right) = 0} 
\label{equi4}
$$
where $T_{.y}$ is the column of the $T$-matrix associated with the reactant complex $y$. Thus, we have to find a vector $\kappa \in Ker{L_\mathscr{R}}$ such that the second summation is also satisfied. The following is a basis for $Ker{L_\mathscr{R}}$.
\[\left( {\begin{array}{*{20}{c}}
1&{ - 1}&1&{ - 1}&1&{ - 1}&1\\
1&0&0&0&0&0&0\\
0&1&0&0&0&0&0\\
0&0&1&0&0&0&0\\
0&0&0&1&0&0&0\\
0&0&0&0&1&0&0\\
0&0&0&0&0&1&0\\
0&0&0&0&0&0&1
\end{array}} \right)\]

\begin{table}
\caption{Summary of values for the Running Example}
\label{tab:values:heck}         
\centering
\begin{tabular}{cccccc}
\noalign{\smallskip}\hline\noalign{\smallskip}
$y\to y'$&  $\kappa_{y \to y'}$ & $\kappa_{y \to y'}e^{T_{.y}\cdot \mu}$  & $\alpha_ik_1$ or $\beta_ik_2$\\
\noalign{\smallskip}\hline\noalign{\smallskip}
 $R_1$   & $1$ &  2.718281828  & 0.466582793\\
 $R_2$ & $1$   &       2.718281828 & 0.859140914\\
$R_3$ & $1$   &       0.367879441 & 0.343292434\\
$R_4$  & $1$  &       0.367879441 & 0.632120559\\
 $R_5$  & $1$  &       1 & 0.294936576\\
 $R_6$  & $1$  & 1 & 0.543080635\\
 $R_7$  & $1$  &      7.389056099 & 0.400860367\\
$R_8$ & $1$   &      7.389056099 &  0.738123111\\
\noalign{\smallskip}\hline
\end{tabular}
\end{table}

Therefore, the original network with reactions $r_1:5X+Y\rightarrow X+3Y$ and $r_2:X+3Y\rightarrow 5X+Y$ with the following kinetics:
\begin{align*}
{K_{{r_1}}}(X,Y) &= {0.466583{X^2}Y + 0.343292X{Y^2} + 0.294937{X^2}{Y^2} + 0.400860{X^3}Y}\\
{K_{{r_2}}}(X,Y) &= {0.859141X + 0.632121Y + 0.543081XY + 0.738123{X^2}}
\end{align*}
admits multiple equilibria $c^{*} = (3.1640,0.5820)$ and $c^{**} = (1.1640,1.5820)$.

\section{Conclusion and Outlook}
\label{sec:conclusion}
The paper established a general method to determine whether a chemical reaction network endowed with poly-PL kinetics (PYK), i.e., a kinetics which is non-negative linear combination of power law kinetics (PLK), has the capacity to admit at least two positive equilibria within a stoichiometric class. By first using a transformation called STAR-MSC that converts a PYK system into a dynamically equivalent PLK system, we analyze the multistationarity capacity of a CRN with PYK system by employing the Multistationarity Algorithm of Hernandez et al. \cite{hmr2019}.  An illustrative example is provided to further clarify the steps of the transformation and algorithm. 

After applying MSA and the fact that a positive equilibrium of a PYK system $(\mathscr N,K)$ is also a positive equilibrium of its transform $(\mathscr N^\ast,K^\ast)$ for the set of rate constants given by the PYK system, we obtain the following:

If there are no rate constants for which $\left ( \mathscr{N}^*,K^* \right )$ has the capacity for multistationarity (i.e., it is monostationary for all stoichiometric classes), then the original system $\left ( \mathscr{N},K \right )$ is also monostationary for all stoichiometric classes.
Otherwise, if after solving for the possible multiple equilibria in the MSA for the PLK system and assuming these equilibria for the original PYK sytem, there exist corresponding rate constants, then the original system admits multiple equilibria in a stoichiometric class.

On-going research works such as those of Fortun et al. \cite{FTJM2020} and Hernandez and Mendoza \cite{HEME2020} also take a similar direction of this paper.  In  \cite{FTJM2020}, the authors used the canonical PL-representation of a PYK system and the STAR-MSC to come up with extended results on generalized mass action systems.  These also include those of M\"uller and Regensburger \cite{muller,MURE2014} on complex equilibria multiplicity and of Boros et al. \cite{BOMR2019} on linear stability to subsets of PY-RDK (i.e., subset of complex factorizable PYK) systems. On the other hand, \cite{HEME2020} studied Hill-type kinetic (HTK) systems associating them with corresponding unique PYK system. This association lead to the identification of subsets of the HTK systems for which recent results on multiplicity and concentration robustness can be applied. The results allow computational approaches for further applications of PYK systems.

\section*{Acknowledgement}
DMM and BSH acknowledge the support of DOST -- SEI (Department of Science and Technology -- Science Education Institute), Philippines for the ASTHRDP Scholarship grant.

\end{document}